\documentclass{amsart}

\usepackage[top=1in,bottom=1in,left=1in,right=1in]{geometry}
\usepackage{amsfonts,
amsmath,
amsthm,
amssymb,
url,
hyperref,
MnSymbol,
color}
\hypersetup{colorlinks=true,
urlcolor=blue,
linkcolor=cyan}

\newcommand{\C}{\ensuremath{{\mathbb{C}}}}
\newcommand{\F}{\ensuremath{{\mathbb{F}}}}

\newcommand{\Q}{\ensuremath{{\mathbb{Q}}}}
\newcommand{\R}{\ensuremath{{\mathbb{R}}}}
\newcommand{\Z}{\ensuremath{{\mathbb{Z}}}}
\newcommand{\lv}{\ensuremath{\left\vert}}
\newcommand{\rv}{\ensuremath{\right\vert}}
\newcommand{\lp}{\ensuremath{\left(}}
\newcommand{\rp}{\ensuremath{\right)}}
\newcommand{\lb}{\ensuremath{\left\{}}
\newcommand{\rb}{\ensuremath{\right\}}}
\newcommand{\ra}{\ensuremath{\right\rangle}}
\newcommand{\la}{\ensuremath{\left\langle}}

\DeclareMathOperator{\cl}{cl}
\DeclareMathOperator{\Pic}{Pic}
\DeclareMathOperator{\Frob}{Frob}
\DeclareMathOperator{\coker}{coker}
\DeclareMathOperator{\Disc}{Disc}
\DeclareMathOperator{\Fix}{Fix}
\DeclareMathOperator{\Surj}{Surj}
\DeclareMathOperator{\Inj}{Inj}
\DeclareMathOperator{\Hom}{Hom}

\DeclareMathOperator{\GSp}{GSp}
\DeclareMathOperator{\Sp}{Sp}
\DeclareMathOperator{\GL}{GL}
\DeclareMathOperator{\Id}{Id}

\DeclareMathOperator{\Mat}{Mat}
\DeclareMathOperator{\Aut}{Aut}
\DeclareMathOperator{\sub}{sub}
\DeclareMathOperator{\rank}{rank}

\makeatletter
\def\imod#1{\allowbreak\mkern10mu({\operator@font mod}\,\,#1)}
\makeatother

\theoremstyle{plain}
\newtheorem{clheur}{Heuristic}[subsection]
\newtheorem{mallec}[clheur]{Conjecture}
\newtheorem{FWheur}[clheur]{Heuristic}
\newtheorem{Aheur}[clheur]{Heuristic}
\newtheorem{Athm}[clheur]{Theorem}
\newtheorem{prettyfinal}[clheur]{Theorem}
\newtheorem{haar}[clheur]{Lemma}
\newtheorem{dual}[clheur]{Lemma}
\newtheorem{dual2}[clheur]{Lemma}
\newtheorem{orb}[clheur]{Lemma}

\newtheorem{bij}[clheur]{Lemma}
\newtheorem{moments}[clheur]{Corollary}
\newtheorem{burn}[clheur]{Lemma}
\newtheorem{deter}[clheur]{Proposition}
\newtheorem{infi}[clheur]{Corollary}

\newtheorem{product}[clheur]{Lemma}
\newtheorem{converge}[clheur]{Lemma}
\newtheorem{final}[clheur]{Theorem}
\newtheorem{finalc}[clheur]{Corollary}
\newtheorem{trivial}[clheur]{Corollary}
\newtheorem{finalm}[clheur]{Theorem}

\theoremstyle{definition}
\newtheorem{Not}[clheur]{Notation}
\newtheorem{Not21}[clheur]{Definition}
\newtheorem{Not2}[clheur]{Notation}
\newtheorem{Not22}[clheur]{Notation}
\newtheorem{Not23}[clheur]{Notation}
\newtheorem{Not3}[clheur]{Notation}
\newtheorem{Not4}[clheur]{Notation}

\newtheorem{Not6}[clheur]{Notation}

\theoremstyle{remark}
\newtheorem{goa}[clheur]{Goal}
\newtheorem{momentsn}[clheur]{Remark}
\newtheorem{subcos}[clheur]{Note}
\newtheorem{amalgam}[clheur]{Remark}
\newtheorem{forms}[clheur]{Note}
\newtheorem{subt}[clheur]{Note}
\newtheorem{nutime}[clheur]{Note}

\title{Random matrices, the Cohen-Lenstra heuristics, and roots of unity}
\author{Derek Garton}
\date{May 22, 2014}

\begin{document}
\maketitle

\section{Introduction} \label{intro}

\subsection{Cohen-Lenstra-Martinet and Malle} \label{CLMM}

In~\cite{CL}, Cohen and Lenstra presented their famous heuristic principle concerning the distribution of ideal class groups of quadratic number fields.
\begin{clheur}[Cohen and Lenstra, 1984] \label{clheur}
For any odd prime $\ell$, a finite abelian $\ell$-group should appear as the $\ell$-Sylow subgroup of the ideal class group of an imaginary quadratic extension of $\Q$ with frequency inversely proportional to the order of its automorphism group.
\end{clheur}
\noindent With a bit more notation, we can reframe this heuristic.
Let $\mathcal{G}$ be the poset of isomorphism classes of finite abelian $\ell$-groups and for any number field $K$, let $\cl{(K)}$ denote the ideal class group of $K$.
For any group $G$, let $G\left[\ell^\infty\right]$ denote its $\ell$-Sylow subgroup.
Now, since $\sum_{A\in\mathcal{G}}{1/\lv\Aut{A}\rv}=\prod_{i=1}^{\infty}{(1-\ell^{-i})^{-1}}$ (a fact first proved by Hall in~\cite{Hall}), the map from $\mathcal{G}\to\R$ given by $A\mapsto\lv\Aut{A}\rv^{-1}\prod_{i=1}^{\infty}{\lp1-\ell^{-i}\rp}$ defines a discrete probability distribution on $\mathcal{G}$.
\hyperref[clheur]{Heuristic~\ref*{clheur}} is the claim that the statistics of this distribution match the statistics of $\ell$-Sylow subgroups of imaginary quadratic extensions (when ordered by fundamental discriminant).
In other words, \hyperref[clheur]{Heuristic~\ref*{clheur}} is equivalent to the following assertion: for any $A\in\mathcal{G}$,
\[
\lim_{X\to\infty}
{\frac{\lv\lb0\leq D\leq X
\,\,\Big\vert\,\,\substack{-D\text{ a fundamental discriminant}\\
\cl{\lp\Q\lp\sqrt{-D}\rp\rp}\left[\ell^\infty\right]\simeq A}\rb\rv}
{\lv\lb0\leq D\leq X
\mid\substack{-D\text{ a fundamental discriminant}}\rb\rv}}
=\frac{1}{\lv\Aut{A}\rv}\prod_{i=1}^{\infty}{\lp1-\ell^{-i}\rp}.
\]
(We remark that this assertion remains unproven; in fact, the above limit is not even known to exist.)
This heuristic explains many previously observed tendencies of class groups of imaginary quadratic fields, such as: their orders should be divisible by three with probability
\[
1-\prod_{i=1}^{\infty}{\lp1-3^{-i}\rp}
=\frac{1}{3}+\frac{1}{9}+\cdots
\approx.44.
\]

In 1990, Cohen and Martinet~\cite{CM} extended their heuristics to include relative class groups of finite extensions of arbitrary number fields, placing different distributions on $\mathcal{G}$ depending on properties of the family of extensions they study.
Once again, they proved that these distributions imply many numerical observations, thereby obtaining a new family of conjectures.
(Recall that relative ideal class groups are defined as follows: if $K/K_0$ is an extension of number fields, the relative class group $\cl{\lp K/K_0\rp}$ is the kernel of the norm map $\text{N}_{K/K_0}:\cl{(K)}\to\cl{\lp K_0\rp}$.)

Recently, however, Malle~\cite{Mal0} presented new computational data that called into question some of the Cohen-Lenstra-Martinet conjectures.
For example, he studied the 3-parts of the relative class groups of quadratic extensions of $\Q\lp\sqrt{-3}\rp$, which has third roots of unity.
Cohen-Lenstra-Martinet predicted that the class numbers of such extensions should be coprime to 3 with probability
\[
\prod_{i=2}^\infty{\lp1-3^{-i}\rp}\approx.840.
\]
However, when Malle computed the class numbers of the first million of these extensions with discriminant at least $10^{20}$, he discovered that the proportion of them with class number coprime to 3 was about .852.
He conjectured that the proportion of all such class groups that have class number coprime to 3 should be exactly
\[
\frac{4}{3}\prod_{i=1}^\infty{\lp1+3^{-i}\rp^{-1}}\approx.852,
\]
which is in much better agreement with his data.
Two years later, in~\cite{Mal}, Malle presented more computational evidence calling into question more Cohen-Lenstra-Martinet conjectures, once again when there are $\ell$th roots of unity in the base field.
In this paper, he also presented a new family of distributions on $\mathcal{G}$ to describe relative class groups when the base field of the extension has $\ell$th roots of unity but not $\ell^2$th roots of unity (see Conjecture 2.1 in~\cite{Mal}).
These distributions on $\mathcal{G}$ imply rank statistics that seem to be a much better fit for his new data.
A special case of his conjecture is the following:

\begin{mallec}[Malle, 2010] \label{mallec}
Suppose that $A\in\mathcal{G}$ and that $A$ has $\ell$-rank $r$.
Let $K_0$ be a number field with $\ell$th but not $\ell^2$th roots of unity.
Let $\mathcal{S}$ be the set of quadratic extensions $K/K_0$ with a fixed signature (with fixed relative unit rank $u$).
Then
\[
\lim_{X\to\infty}
{\frac{\lv\lb K\in\mathcal{S}
\,\,\big\vert\,\,\substack{\lv\Disc{K}\rv\leq X,\\
\cl{\lp K/K_0\rp}\left[\ell^{\infty}\right]\simeq A}\rb\rv}
{\lv\lb K\in\mathcal{S}
\mid\substack{\lv\Disc{K}\rv\leq X}\rb\rv}}
=\frac{\prod_{i=u+1}^{u+r}{\lp\ell^i-1\rp}}
{\ell^{r(u+1)}\lv A\rv^u\lv\Aut{A}\rv}
\cdot\prod_{i=u+1}^{\infty}{\lp1+\ell^{-i}\rp^{-1}}.
\]
\end{mallec}

In this paper, we study a random matrix model of ideal class groups of function fields when the base field has $\ell$th roots of unity (i.e., the function field analog of the situation Malle studies in \hyperref[mallec]{Conjecture~\ref*{mallec}} ).
We compute the distributions on $\mathcal{G}$ given by this matrix model in two cases (see \hyperref[final]{Theorem~\ref*{final}}):
in the case when base field has $\ell$th roots of unity but not $\ell^2$th roots of unity and in the case when the base field with $\ell^2$th roots of unity but not $\ell^3$th roots of unity.
In the former case, our distribution matches the distribution proposed by Malle.
Moreover, we compute all the moments of the distribution given by this matrix model in the general case when the base case has $\ell^\xi$th but not $\ell^{\xi+1}$th roots of unity for any $\xi\in\Z^{>0}$ (see~\hyperref[moments]{Corollary~\ref*{moments}}).

The work in this paper is based on my 2012 PhD dissertation~\cite{GD}.
The matrix distributions were computed independently in the 2014 PhD dissertation of M. Adam~\cite{AdamD} as well as in his 2014 paper~\cite{AdamP}.
They are also used in a recent paper of Adam and Malle~\cite{AM}.

\subsection{The function field case} \label{FFC}

Complementing the work described in \hyperref[CLMM]{Section~\ref{CLMM}}, investigators have been studying analogous phenomena in function fields defined over finite fields.
In 1989, Friedman and Washington (in~\cite{FW}) addressed the case of quadratic extensions of the field $\F_{p^n}(t)$ for a prime $p\neq2$ and $n\in\Z^{>0}$.
More precisely, if $f(t)\in\F_{p^n}[t]$ is monic with distinct roots of degree $2g+1$, let $C_f$ be the hyperelliptic curve (defined over $\F_{p^n}$) of genus $g$ given by $y^2=f(t)$.
Note that the curve $C_f$ has exactly one point at infinity, just as imaginary quadratic extensions of $\Q$ have exactly one place at infinity.
Thus, $\Pic^0_{\F_{p^n}}{\lp C_f\rp}$ is isomorphic to the ideal class group of the field extension $\mathbb{F}_{p^n}(t)[\sqrt{f(t)}]/\mathbb{F}_{p^n}(t)$.

To study these groups, Friedman and Washington introduced a new heuristic principle, one that comes from the geometry of hyperelliptic curves over finite fields.
Specifically, for $f(t)\in\F_{p^n}[t]$ monic with distinct roots of degree $2g+1$, let $T_{\ell}\lp C_f\rp$ be the $\ell$-adic Tate module of $C_f$, which is a free $2g$-dimensional $\Z_\ell$-module.
In addition, let $\Frob_{p^n}$ be the $p^n$-power Frobenius map acting on $T_\ell\lp C_f\rp$.
Thinking of $\Frob_{p^n}$ as a matrix over $\Z_\ell$, it is well-known that $\coker{\lp\Id-\Frob_{p^n}\rp}$ is isomorphic to the $\ell$-Sylow subgroup of $\Pic^0_{\F_{p^n}}{\lp C_f\rp}$ (see the appendix of~\cite{FW} for a proof of this fact).
Friedman and Washington conjectured that the statistics of $\ell$-Sylow subgroups of ideal class groups of quadratic extensions of $\F_{p^n}(t)$ match the statistics of $\ell$-adic matrices.
Specifically, if we let
\[
F(g,p^n,\ell,A)
:=\frac{\lv\lb f(t)\in\F_{p^n}[t]
\,\Big\vert\,\substack{f\text{ monic with distinct roots}\\
\deg{f}=2g+1,
\,\Pic^0_{\F_{p^n}}{\lp C_f\rp}[\ell^{\infty}]\simeq A}\rb\rv}
{\lv\lb f(t)\in\F_{p^n}[t]
\,\Big\vert\,\substack{f\text{ monic with distinct roots}\\
\deg{f}=2g+1}\rb\rv},
\]
then they proposed the following:
\begin{FWheur}[Friedman and Washington, 1989] \label{FWheur}
If $A\in\mathcal{G}$, then
\[
\lim_{g\to\infty}
{F(g,p^n,\ell,A)}
=\lim_{g\to\infty}
{\alpha_{2g}\lp\lb\phi\in\Mat_{2g}{\lp\Z_\ell\rp}
\mid\coker{(\Id-\phi)}\simeq A\rb\rp},
\]
where $\alpha_{2g}$ is the normalized Haar measure on $\Mat_{2g}{\lp\Z_\ell\rp}$.
\end{FWheur}
\noindent (See Sections~\ref{notdef} and~\ref{THM} for more details on Haar measures.)
A few years later, Katz and Sarnak~\cite{KS} vastly extended the philosophy of considering the action of Frobenius as a random matrix, especially when the size of the base field is large.
Friedman and Washington show that the limit on the right hand side of~\hyperref[FWheur]{Heuristic~\ref*{FWheur}} exists, and that it defines exactly the same distribution on $\mathcal{G}$ as Cohen-Lenstra's original heuristic for imaginary quadratic extensions of \Q.
However, just as the work of Malle calls into question the appropriateness of certain Cohen-Lenstra-Martinet distributions, it also calls into question the appropriateness of Friedman-Washington's proposed distribution.
Indeed, note the Friedman-Washington heuristic does not depend at all on the presence of $\ell$th roots of unity in the base field $\F_{p^n}(t)$, while Malle's work suggests that distributions modeling $\ell$-Sylow subgroups of class groups ought to change in the presence presence of $\ell$th roots of unity.
Thus, the new data of Malle suggests that \hyperref[FWheur]{Heuristic~\ref*{FWheur}} might be flawed when $\F_{p^n}(t)$ has $\ell$th roots of unity.

A possible explanation for this flaw is that $\Frob_{p^n}$ is a symplectic similitude with respect to the Weil pairing on $T_\ell(C_f)$.
Indeed, it scales the Weil pairing by $p^n$, so when considered as a matrix, $\Frob_{p^n}\in\GSp_{2g}^{(p^n)}{(\Z_\ell)}$.
(See \hyperref[notdef]{Section~\ref*{notdef}} for more details on this notation.)
Since the presence of $\ell$th roots of unity in $\F_{p^n}(t)$ depends on the congruence class of $p^n\pmod{\ell}$, the set of symplectic similitudes that scale the Weil pairing by $p^n$ does indeed change change when $\F_{p^n}(t)$ has $\ell$th roots of unity.
These facts led Friedman and Washington (and Achter~\cite{A2}) to suggest
\begin{Aheur} \label{Aheur}
If $A\in\mathcal{G}$, then
\[
\lim_{g\to\infty}
{F(g,p^n,\ell,A)}
=\lim_{g\to\infty}
{\mu_{2g}^{(p^n)}\lp\lb\phi\in\GSp^{(p^n)}_{2g}{\lp\Z_\ell\rp}
\mid\coker{(\Id-\phi)}\simeq A\rb\rp},
\]
where $\mu_{2g}^{(p^n)}$ is the unique normalized multiplicative Haar measure on $\Sp_{2g}{\lp\Z_\ell\rp}$ translated to $\GSp_{2g}^{(p^n)}{\lp\Z_\ell\rp}$.
\end{Aheur}
\noindent(Again, see Sections~\ref{notdef} and~\ref{THM} for more details on Haar measures.)
Friedman and Washington hoped that this new heuristic would turn out to describe the same distribution as \hyperref[FWheur]{Heuristic~\ref*{FWheur}}, but Achter~\cite{A1} proved that
\begin{align*}
\lim_{g\to\infty}
&{\mu_{2g}^{(p^n)}\lp\lb\phi\in\GSp^{(1)}_{2g}{\lp\Z_\ell\rp}
\mid\coker{(\Id-\phi)}\simeq\{0\}\rb\rp}\\
&\hspace{110px}\neq\lim_{g\to\infty}
{\alpha_{2g}\lp\lb\phi\in\Mat_{2g}{\lp\Z_\ell\rp}
\mid\coker{(\Id-\phi)}\simeq\{0\}\rb\rp},
\end{align*}
revealing that this was not the case, providing an early indication of the importance of the  presence of $\ell$th roots of unity in the base field.
In~\cite{A2}, Achter uses work of Katz-Sarnak~\cite{KS} to prove a revised version of~\hyperref[FWheur]{Heuristic~\ref*{FWheur}}:
\begin{Athm}[Achter, 2008] \label{Athm}
If $A\in\mathcal{G}$, then
\[
\lim_{p^n\to\infty}
{\lv F(g,p^n,\ell,A)
-\mu_{2g}^{(p^n)}\lp\lb\phi\in\GSp^{(p^n)}_{2g}{\lp\Z_\ell\rp}
\mid\coker{(\Id-\phi)}\simeq A\rb\rp\rv}
=0.
\]
\end{Athm}
\noindent We remark that this limit in \hyperref[Athm]{Theorem~\ref*{Athm}} leaves $g$ fixed while letting $p^n$ increase, whereas the limit in \hyperref[Aheur]{Heuristic~\ref*{Aheur}} does the opposite.

The work of Ellenberg, Venkatesh, and Westerland~\cite{EVW} uses the topology of Hurwitz spaces to study \hyperref[Aheur]{Heuristic~\ref*{Aheur}}.
One consequence of their work is that
\[
\lim_{g\to\infty}
{\lim_{\substack{p^n\to\infty\\
p^n\nequiv1\imod{\ell}}}
{F(g,p^n,\ell,A)}}
=\frac{1}{\lv\Aut{A}\rv}\prod_{i=1}^{\infty}{\lp1-\ell^{-i}\rp}.
\]
Since $p^n\equiv1\pmod{\ell}$ exactly when $\F_{p^n}(t)$ has $\ell$th roots of unity, this result only addresses the case when the base field does not have $\ell$th roots of unity (and only when $p^n\to\infty$).
The remaining case is when $p^n\equiv1\pmod{\ell}$; that is, the case where there are $\ell$th roots of unity in the base field.
\hyperref[mallec]{Conjecture~\ref*{mallec}} suggests that a different distribution is needed to describe this case.
In fact, \hyperref[finalc]{Corollary~\ref*{finalc}} gives such a distribution.
Using Achter's result, \hyperref[Athm]{Theorem~\ref*{Athm}}, \hyperref[finalc]{Corollary~\ref*{finalc}} implies the following theorem:

\begin{prettyfinal} \label{prettyfinal}
If $A$ is a finite abelian $\ell$-group with $\ell$-rank $r$ and $\ell^2$-rank $s$, then
\begin{align*}
\lim_{g\to\infty}&{
\lim_{\substack{p^n\to\infty\\
p\nequiv1\imod{\ell^\xi},\\
p\nequiv1\imod{\ell^{\xi+1}}}}
{F(g,p^n,\ell,A)}}\\
&\phantom{======}=\begin{cases}
\ell^{\frac{r(r-1)}{2}}\cdot\lp\ell^{-1};\ell^{-1}\rp_r
\cdot\frac{\prod_{i=1}^{\infty}{(1+\ell^{-i})^{-1}}}{\lv\Aut{A}\rv^{-1}},
&\text{if }\xi=1\\
\ell^{\frac{r(r-1)}{2}+\frac{s(s-1)}{2}}
\cdot\lp\ell^{-1};\ell^{-1}\rp_s
\cdot\lp\ell^{-1};\ell^{-2}\rp_{\lceil \frac{r-s}{2}\rceil}
\cdot\frac{\prod_{i=1}^{\infty}{(1+\ell^{-i})^{-1}}}{\lv\Aut{A}\rv^{-1}},
&\text{if }\xi=2,
\end{cases}
\end{align*}
where $\lp\ell^{-1};\ell^{-j}\rp_k$ is the $\ell^{-j}$-Pochhammer symbol, defined for any $j\in\Z^{>0}$ and $k\in\Z^{\geq0}$ (see \hyperref[Not6]{Notation~\ref*{Not6}}).
\end{prettyfinal}

\noindent \hyperref[prettyfinal]{Theorem~\ref*{prettyfinal}} extends \hyperref[mallec]{Conjecture~\ref*{mallec}} by including the case where $\F_{p^n}(t)$ has $\ell^2$th roots of unity but not $\ell^3$th roots of unity.
We remark that since imaginary hyperelliptic curves have only one place at infinity, the function field version of \hyperref[mallec]{Conjecture~\ref*{mallec}} should set $u=0$; making this substitution in \hyperref[mallec]{Conjecture~\ref*{mallec}} yields the $\xi=1$ case of \hyperref[prettyfinal]{Theorem~\ref*{prettyfinal}}.

\section{Prelinimaries}

\subsection{Notation and definitions} \label{notdef}

As above, let $\ell$ be an odd prime and let $\mathcal{G}$ be the poset of isomorphism classes of finite abelian $\ell$-groups, with the relation $\left[A\right]\leq\left[B\right]$ if and only if there exists an injection $A\hookrightarrow B$.
(For notational simplicity, we will conflate finite abelian $\ell$-groups and the equivalence classes containing them.)
For any $A\in\mathcal{G}$, we denote the exponent of $A$ by $\exp{A}$.
If $i\in\Z^{>0}$, let
\[
\rank_{\ell^i}{A}:=\dim_{\F_\ell}{\lp\ell^{i-1}A/\ell^iA\rp}.
\]
We will abbreviate $\rank_\ell{A}$ by $\rank{A}$.
If $r_1,\ldots,r_{i-1}\in\Z^{>0}$ and $r_i\in\Z^{\geq0}$, let $\mathcal{G}(r_1,\ldots,r_i)$ be the following subposet of $\mathcal{G}$:
\[
\mathcal{G}(r_1,\ldots,r_i)
:=\lb A\in\mathcal{G}\mid\rank_{\ell^j}{A}=r_j\text{ for all }j\in\lb1,\ldots,i\rb\rb.
\]
Next, for any $\rho\in\Z^{>0}$, set $R_\rho=\Z_\ell/\ell^\rho\Z_\ell\simeq\Z/\ell^\rho\Z$.
For any $g,\rho\in\Z^{>0}$, let $\la\cdot,\cdot\ra_{2g,\rho}$ be a fixed choice of symplectic form on $(R_\rho)^{2g}$; that is, $\la\cdot,\cdot\ra_{2g,a}:(R_\rho)^{2g}\times(R_\rho)^{2g}\to(R_\rho)$ is $R_\rho$-bilinear, alternating, and nondegenerate.
By nondegenerate, we mean that the matrix associated to $\la\cdot,\cdot\ra_{2g,\rho}$ is invertible; see Theorem~III.2 of~\cite{McD} for more details on symplectic spaces.
Similarly, let $\la\cdot,\cdot\ra_{2g}$ be a choice of symplectic form on $\lp\Z_\ell\rp^{2g}$.
For any ring $R$ and any $g\in\Z^{>0}$, if $R^{2g}$ has a symplectic form $\la\cdot,\cdot\ra$, then the \emph{symplectic group of $R$} is
\[
\Sp_{2g}{(R)}
\simeq\Sp{\lp R^{2g},\la\cdot,\cdot\ra\rp}\\
=\lb\phi\in\GL{\lp R^{2g}\rp}
\mid\la \phi(x),\phi(y)\ra=\la x,y\ra\text{ for all }x,y\in R^{2g}\rb.
\]
Note that a different choice of symplectic form on $R^{2g}$ yields an isotropic space, so the choice is immaterial (see page 188 of~\cite{McD} for more details).
Similarly, the \emph{group of symplectic similitudes of $R$} is
\[
\GSp_{2g}{(R)}
\simeq\GSp{\lp R^{2g},\la\cdot,\cdot\ra\rp}
=\lb\phi\in\GL{\lp R^{2g}\rp}
\,\Big\vert\,\substack{\text{there exists }m(\phi)\in R^{\times}\text{ such that}\\
\la\phi(x),\phi(y)\ra=m(\phi)\cdot\la x,y\ra
\text{ for all }x,y\in R^{2g}}\rb.
\]
For concreteness, we will always assume that the rings $(R_\rho)^{2g}$ and $(\Z_\ell)^{2g}$ are equipped with the forms $\la\cdot,\cdot\ra_{2g,\rho}$ and $\la\cdot,\cdot\ra_{2g}$ fixed above.
The map $m:\GSp_{2g}{\lp R\rp}\to R^{\times}:\phi\mapsto m(\phi)$ given above is a homomorphism called the \emph{multiplier map}, and the element $m(\phi)\in R^{\times}$ is called the \emph{multiplier} of~$\phi$.
For any $g\in\Z^{>0}$, let $\mu_{2g}$ be the unique normalized Haar measure on $\Sp_{2g}{\lp\Z_{\ell}\rp}$, noting that this measure is invariant under both left and right multiplication since $\Sp_{2g}{\lp\Z_{\ell}\rp}$ is a unimodular group.
Finally, for any $g\in\Z^{>0}$ and any unit $x$ in a ring $R$, let $\GSp_{2g}^{(x)}{\lp R\rp}=m^{-1}(x)$.

Now, for any $x\in\lp\Z_{\ell}\rp^{\times}$ and $\phi\in\GSp_{2g}^{(x)}{\lp \Z_{\ell}\rp}$ we define a measure $\mu_{2g}^{(x)}$ on $\GSp_{2g}^{(x)}{\lp \Z_{\ell}\rp}$ as follows: for any $\mu_{2g}$-measurable subset $S\subseteq\Sp_{2g}{\lp \Z_{\ell}\rp}$, define
\[
\mu_{2g}^{(x)}\lp S\phi\rp:=\mu_{2g}\lp S\rp.
\]
We remark that this measure is independent of the choice $\phi\in\GSp_{2g}^{(x)}{\lp\Z_{\ell}\rp}$.
Indeed, given some other $\psi\in\GSp_{2g}^{(x)}{\lp \Z_{\ell}\rp}$ there exists a unique $\phi_\psi\in\Sp_{2g}{\lp\Z_{\ell}\rp}$ such that $\phi_\psi\phi=\psi$; ie, $S\psi=S\phi_\psi\phi$.
Since $\mu_{2g}$ is translation invariant, we know that
\[
\mu_{2g}\lp S\rp=\mu_{2g}\lp S\phi_\psi\rp,
\]
as desired.
Moreover, since $\mu_{2g}$ is translation invariant (by $\Sp_{2g}{\lp\Z_{\ell}\rp}$) and normalized, so is $\mu_{2g}^{(x)}$.
Similarly, for any $\rho\in\Z^{>0}$, let $\nu_{2g,\rho}$ be the unique normalized Haar measure on $\Sp_{2g}{\lp R_\rho\rp}$, and for any $x\in R_\rho^{\times}$, define $\nu_{2g,\rho}^{(x)}$ on $\GSp_{2g}^{(x)}{\lp R_\rho\rp}$ as above.
For any $\rho\in\Z^{>0}$, $x\in R_\rho^{\times}$, and $S\subseteq\GSp_{2g}^{(x)}{\lp R_\rho\rp}$, we know $\nu_{2g,\rho}^{(x)}(S)=\lv S\rv\cdot\lv\Sp_{2g}{\lp R_\rho\rp}\rv^{-1}$, since $\Sp_{2g}{\lp R_\rho\rp}$ is a finite group.
To ease notation, for any $A\in\mathcal{G}$, $g\in\Z^{>0}$, and $x\in\lp\Z_{\ell}\rp^{\times}$, we set
\[
\mu_{2g}^{(x)}\lp A\rp
:=\mu_{2g}^{(x)}\lp\lb\phi\in\GSp_{2g}^{(x)}{\lp\Z_{\ell}\rp}
\,\,\big\vert\,\,\coker{\lp\Id-\phi\rp}\simeq A\rb\rp.
\]
Furthermore, if $\rho\in\Z^{>0}$ and $x\in R_\rho^{\times}$, set
\[
\nu_{2g,\rho}^{(x)}\lp A\rp
:=\nu_{2g,\rho}^{(x)}\lp\lb\gamma\in\GSp_{2g}^{(x)}{\lp R_\rho\rp
\,\,\Big\vert\,\,\coker{\lp\Id-\gamma\rp}\simeq A}\rb\rp.
\]

\subsection{The Haar measures} \label{THM}

The measures defined in \hyperref[notdef]{Section~\ref*{notdef}} have an important relationship, given in the following lemma.

\begin{haar} \label{haar}
Suppose $A\in\mathcal{G}$, $g\in\Z^{>0}$, $x\in\lp\Z_\ell\rp^{\times}$, and $\rho\in\Z^{>0}$.
Let $\overline{\phantom{\vert}\cdot\phantom{\vert}}:\Z_\ell\to R_\rho$ denote reduction mod $\ell^\rho$.
If $\ell^\rho>\exp{A}$, then
\[
\mu_{2g}^{(x)}\lp A\rp=\nu_{2g,\rho}^{(\overline{x})}\lp A\rp.
\]
\end{haar}
\begin{proof}
Choose any $\phi\in\GSp_{2g}^{(x)}{\lp\Z_{\ell}\rp}$.
Then for any measurable $S\subseteq\GSp_{2g}^{(x)}{\lp\Z_{\ell}\rp}$, we know that $\mu_{2g}^{(x)}\lp S\rp=\mu_{2g}^{(x)}\lp S\phi^{-1}\phi\rp=\mu_{2g}\lp S\phi^{-1}\rp$ by the definition of $\mu_{2g}^{(x)}$.
Since $\mu_{2g}$ is invariant under translation, every coset of the kernel of the reduction map $\overline{\phantom{\vert}\cdot\phantom{\vert}}:\Sp_{2g}{\lp\Z_{\ell}\rp}\to R_\rho$ has the same measure; namely,
\[
\left[\Sp_{2g}{\lp\Z_{\ell}\rp}:\ker{\lp\overline{\phantom{i}\cdot\phantom{i}}\rp}\right]^{-1}=\lv\Sp_{2g}{\lp R_\rho\rp}\rv^{-1}.
\]
Moreover, note that if $\psi\in\GSp_{2g}^{(x)}{\lp\Z_{\ell}\rp}$, then $m\lp\overline{\psi}\rp=\overline{m(\psi)}$ and $\coker{\lp\Id-\psi\rp}\simeq A$ if and only if $\coker{\lp\Id-\overline{\psi}\rp}\simeq A$, since $\ell^\rho>\exp{A}$.
The result follows.
\end{proof}

\begin{Not} \label{Not}
Suppose that $g\in\Z^{>0}$ and $\xi\in\Z^{\geq0}$.
For $\rho\in\Z^{>0}$ satisfying $\rho\geq\xi$, we define an important subgroup of $\GSp_{2g}{\lp R_\rho\rp}$:
\[
\GSp^{\la\xi\ra}_{2g}{\lp R_\rho\rp}
:=\lb\gamma\in\GSp_{2g}{\lp R_\rho\rp}
\mid m(\gamma)\in\ell^\xi R_\rho+1\rb.
\]
Note that $\GSp^{\la\rho\ra}_{2g}{\lp R_\rho\rp}=\Sp_{2g}{\lp R_\rho\rp}$ and $\GSp^{\la0\ra}_{2g}{\lp R_\rho\rp}=\GSp_{2g}{\lp R_\rho\rp}$.
For any $A\in\mathcal{G}$, we adopt the suggestive notation:
\[
N_{2g,\rho}^{\la\xi\ra}\lp A\rp
:=\lv\lb\gamma\in\GSp^{\la\xi\ra}_{2g}{\lp R_\rho\rp}
\mid\coker{(\Id-\gamma)}\simeq A\rb\rv
\]
and, if $\rho>\xi$,
\[
\nu_{2g,\rho}^{\la\xi\ra}\lp A\rp
:=\frac{N_{2g,\rho}^{\la\xi\ra}\lp A\rp
-N_{2g,\rho}^{\la\xi+1\ra}\lp A\rp}
{\lv\GSp^{\la\xi\ra}_{2g}{\lp R_\rho\rp}\rv
-\lv\GSp^{\la\xi+1\ra}_{2g}{\lp R_\rho\rp}\rv}.
\]
\end{Not}

\begin{goa} \label{goa}
We can now state the matrix-theoretic analog of the situation about which Malle made Conjecture~2.1.
Following \hyperref[Aheur]{Conjecture~\ref*{Aheur}}, for $A\in\mathcal{G}$, $x\in\lp\Z_{\ell}\rp^{\times}$ and $\xi\in\Z^{>0}$, with $x\equiv1\pmod{\ell^\xi}$ but $x\nequiv1\pmod{\ell^{\xi+1}}$, we must evaluate
\[
\mu_x(A):=\lim_{g\to\infty}
{\mu_{2g}^{(x)}\lp A\rp}.
\]
If we let $\overline{\phantom{\vert}\cdot\phantom{\vert}}:\Z_\ell\to R_\rho$ denote reduction mod $\ell^\rho$, then we know by \hyperref[haar]{Lemma~\ref*{haar}} that this amounts to calculating
\[
\lim_{g\to\infty}
{\nu_{2g,\rho}^{(\overline{x})}\lp A\rp}
\]
for any $\rho\in\Z^{>0}$ satisfying both $\ell^\rho>\exp{A}$ and $\rho>\xi$.
In \hyperref[subcos]{Note~\ref*{subcos}} we will see that for all such $\rho$:
\[
\nu_{2g,\rho}^{(\overline{x})}\lp A\rp=\nu_{2g,\rho}^{\la\xi\ra}\lp A\rp,
\]
so we will turn our attention to computing
\[
\lim_{g\to\infty}
{\nu_{2g,\rho}^{\la\xi\ra}\lp A\rp},
\]
which we compute explicitly for $\xi=1,2$ in \hyperref[finalc]{Corollary~\ref*{finalc}}.
Using Achter's result, \hyperref[Athm]{Theorem~\ref*{Athm}}, we then obtain \hyperref[prettyfinal]{Theorem~\ref*{prettyfinal}} as a corollary.
\end{goa}

\begin{momentsn} \label{momentsn}
Suppose that $x\in\Z_\ell$.
In addition to explicitly computing the distribution $\mu_x:\mathcal{G}\to\R$ if $x\equiv1\pmod{\ell^\xi}$ but $x\nequiv1\pmod{\ell^{\xi+1}}$ for $\xi=1,2$, we also compute the moments of this distribution for any $\xi\in\Z^{>0}$.
Specifically, in \hyperref[moments]{Corollary~\ref*{moments}} we find that if $A\in\mathcal{G}$, then
\[
\sum_{B\in\mathcal{G}}{\lv\Surj{(B,A)}\rv\mu_x(B)}
=\lv\Lambda\lp A/\ell^\xi\rp\rv.
\]
(See \hyperref[Not22]{Notation~\ref*{Not22}} for the definition of $\Lambda:\mathcal{G}\to\Z^{>0}$.)
\end{momentsn}

\section{The symplectic action} \label{TSA}

\subsection{Basic properties} \label{bp}

\begin{Not2} \label{Not2}
For any $A,B\in\mathcal{G}$, let $\Inj{\lp A,B\rp}$ and $\Surj{\lp A,B\rp}$ be the set of injective homomorphisms and surjective homomorphisms from $A$ to $B$.
\end{Not2}

In what follows, we will consider either injections or surjections (as well as either kernels or cokernels) depending on which is more convenient at the time.
The next two lemmas justify this shifting point of view.

\begin{dual} \label{dual}
Suppose that $A\in\mathcal{G}$, $g,\rho\in\Z^{>0}$, and $\xi\in\Z^{\geq0}$.
If $\rho\geq\xi$, then $\GSp^{\la \xi\ra}_{2g}{(R_\rho)}$ acts on $\Inj{(A,(R_\rho)^{2g})}$ and $\Surj{((R_\rho)^{2g},A)}$ by post- and pre-composition, respectively.
The number of orbits of these actions are the same.
\end{dual}
\begin{proof}
If $\ell^\rho<\exp{A}$, the result is trivial, so suppose $\ell^\rho\geq\exp{A}$.
In this case, we can think of $A$ as an $R_\rho$-module.
Moreover, we know that $R_\rho$ is an injective $R_\rho$-module by Baer's criterion, so the functor
\[
\lp\cdot\rp^\vee
:=\Hom(\,\cdot\,,R_\rho):R_\rho\mathtt{-mod}\to R_\rho\mathtt{-mod}
\]
is exact.
Thus, for any $\gamma\in\GSp^{\la\xi\ra}_{2g}{(R_\rho)}$:
\begin{gather*}
f,h\in\Inj{(A,(R_\rho)^{2g})}
\text{ with }\gamma\circ f=h\\
\text{ if and only if }\\
f^{\vee},h^{\vee}\in\Surj{(((R_\rho)^{2g})^\vee,A^\vee)}
\text{ with }f^\vee\circ\gamma^{\vee}=h^{\vee}.
\end{gather*}
After choosing $R_\rho$-bases for $(R_\rho)^{2g}$ and $A$, it is easy to see that $((R_\rho)^{2g})^\vee\simeq(R_\rho)^{2g}$, $A^\vee\simeq A$, and $\gamma^\vee=\gamma^\top\in\GSp^{\la\xi\ra}_{2g}{(R_\rho)}$, giving the result.
\end{proof}

The number orbits of the action described above turn out to be very important, so we bestow a name upon them:

\begin{Not21} \label{Not21}
Suppose that $A\in\mathcal{G}$, $g,\rho\in\Z^{>0}$, and $\xi\in\Z^{\geq0}$.
If $\rho\geq\xi$, let $o^{A,\la\xi\ra}_{2g,\rho}$ be the number of orbits of $\GSp^{\la\xi\ra}_{2g}{(R_\rho)}$ acting on $\Inj{\lp A,(R_\rho)^{2g}\rp}$ or $\Surj{\lp (R_\rho)^{2g},A\rp}$.
\end{Not21}

\begin{dual2} \label{dual2}
For $A,g,\rho,\xi$ as above:
\[
N_{2g,\rho}^{\la\xi\ra}\lp A\rp
=\lv\lb\gamma\in\GSp^{\la\xi\ra}_{2g}{\lp R_\rho\rp}
\mid\ker{(\Id-\gamma)}\simeq A\rb\rv.
\]
\end{dual2}
\begin{proof}
As in \hyperref[dual]{Lemma~\ref*{dual}}, this follows from the exactness of $\lp\cdot\rp^\vee$.
Note that for any $\gamma\in\GSp_{2g}^{\la\xi\ra}{(R_\rho)}$:
\[
\lp\coker{\lp\Id-\gamma\rp}\rp^{\vee}
=\ker{\lp\lp\Id-\gamma\rp^{\vee}\rp}
=\ker{\lp\Id-\gamma^{\top}\rp},
\]
giving the result.
\end{proof}

In \hyperref[goa]{Goal~\ref*{goa}}, we turned our attention from the measures of cosets of the symplectic group to subgroups of the group of symplectic similitudes.
The following note justifies this turn.

\begin{subcos} \label{subcos}
Suppose that $A\in\mathcal{G}$, $g\in\Z^{>0}$, $x\in\Z_{\ell}$ and $\xi\in\Z^{>0}$, with $x\equiv1\pmod{\ell^\xi}$ but $x\nequiv1\pmod{\ell^{\xi+1}}$.
If $\rho\in\Z^{>0}$ satisfies $\rho>\xi$, then
\[
\nu_{2g,\rho}^{(\overline{x})}(A)=\nu_{2g,\rho}^{\la\xi\ra}(A)
\]
\end{subcos}
\begin{proof}
This amounts to showing that if $x,y\in R_\rho$ such that $x\equiv y\equiv1\pmod{\ell^\xi}$ but neither $x$ nor $y$ is equivalent to $1\pmod{\ell^{\xi+1}}$, then
\[
\nu_{2g,\rho}^{\lp\overline{x}\rp}(A)=\nu_{2g,\rho}^{\lp\overline{y}\rp}(A).
\]
By our assumptions on $x$ and $y$, there exists some $m_0$ such that that $\ell\nmid m_0$ and $\overline{x}^{\hspace{1px}m_0}=\overline{y}$.
Choose some $m$ in the arithmetic progression $\lb m_0+\ell^{\rho-\xi}j\rb_{j=0}^\infty$ such that
\[
\gcd{\lp m,\lv\GSp_{2g}{(R_\rho)}\rv\rp}=1,
\]
and choose $k$ such that $mk\equiv1\pmod{\lv\GSp_{2g}{(R_\rho)}\rv}$.
Now, the map
\begin{align*}
\lp\cdot\rp^m:\GSp_{2g}^{\lp\overline{x}\rp}{(R_\rho)}&\to\GSp_{2g}^{\lp\overline{y}\rp}{(R_\rho)}\\
\gamma&\mapsto\gamma^m
\end{align*}
is bijective with inverse $\lp\cdot\rp^k$.
Moreover, for any $z\in(R_\rho)^{2g}$ and any $\gamma\in\GSp_{2g}^{\lp\overline{x}\rp}{(R_\rho)}$, it is clear that $\gamma z=z$ if and only if $\gamma^mz=z$.
Thus, we obtain
\[
\lv\lb\gamma\in\GSp_{2g}^{\lp\overline{x}\rp}{(R_\rho)}
\mid\ker{\lp\Id-\gamma\rp}\simeq A\rb\rv
=\lv\lb\gamma\in\GSp_{2g}^{\lp\overline{y}\rp}{(R_\rho)}
\mid\ker{\lp\Id-\gamma\rp}\simeq A\rb\rv,
\]
and we conclude by \hyperref[dual2]{Lemma~\ref*{dual2}}.
\end{proof}

\subsection{Orbit counting} \label{oc}

\begin{Not22} \label{Not22}
For any $A\in\mathcal{G}$, let $\Lambda\lp A\rp$ be the set of alternating bilinear forms on $A$ thought of as a $\lp\Z/\exp{(A)}\rp$-module.
\end{Not22}

\begin{forms} \label{forms}
Suppose that $A=\Z/\ell^{\alpha_1}\oplus\cdots\oplus\Z/\ell^{\alpha_r}$ with $\alpha_1\geq\cdots\geq\alpha_r>0$.
Then
\[
\lv\Lambda\lp A\rp\rv=\ell^{\sum_{i=2}^r{(i-1)\alpha_i}}
\]
\end{forms}
\begin{proof}
Let $\lb\mathbf{e}_i\rb_{i=1}^r$ be an $\lp\Z/\exp{(A)}\rp$-basis for $A$ such that $\mathbf{e}_i$ has order $\ell^{\alpha_i}$ for all $i\in\lb1,\ldots,r\rb$.
Every alternating bilinear form $\la\cdot,\cdot\ra$ on $A$ corresponds to an antisymmetric matrix $\lp\la\mathbf{e}_i,\mathbf{e}_j\ra\rp\in\Mat_{r\times r}{\lp\Z/\exp{(A)}\rp}$.
Moreover, any antisymmetric matrix $\lp a_{ij}\rp\in\Mat_{r\times r}{\lp\Z/\exp{(A)}\rp}$ corresponds to an alternating bilinear form on $A$, as long it has 0's along its main diagonal and $\ell^{\alpha_j}a_{ij}=0$ whenever $i<j$ (the $i>j$ case follows from the the fact that $\lp a_{ij}\rp$ is antisymmetric).
There are $\ell^{\alpha_j}$ such elements of $\Z/\exp{(A)}$, so the result follows.
\end{proof}

\begin{orb} \label{orb}
Suppose that $r\in\Z^{\geq0}$, $A\in\mathcal{G}(r)$, $g,\rho\in\Z^{>0}$, and $\xi\in\Z^{\geq0}$.
If $\ell^\rho\geq\exp{A}$, $\rho\geq\xi$, and $2g\geq r$, then
\[
o^{A,\la\xi\ra}_{2g,\rho}
\leq\ell^{-(\rho-\xi)}\lv\Lambda(A)\rv+
(\ell-1)\sum_{i=0}^{\rho-\xi-1}{\ell^{-(i+1)}}{\lv\Lambda\lp A/\ell^{\xi+i}\rp\rv}.
\]
Furthermore, when $g\geq r$, the upper bound above is an equality.
(In particular, $o^{A,\la\xi\ra}_{2g,\rho}$ is independent of $g$ for large enough $g$.)
\end{orb}
As pointed out in \hyperref[goa]{Goal~\ref*{goa}}, we need only calculate
\[
\lim_{g\to\infty}{\nu_{2g,\rho}^{\la\xi\ra}(A)}.
\]
Despite this fact, the inequality for small $g$ in \hyperref[orb]{Lemma~\ref*{orb}} does indeed turn out to be useful.
This is due to the fact that $\nu_{2g,\rho}^{\la\xi\ra}(A)$ can be expressed as a sum of orbit data for finite abelian groups of rank up to $2g$. 
(See \hyperref[infi]{Corollary~\ref*{infi}}.)

\begin{proof}[Proof of the Lemma]
The result is obviously true when $r=0$, so suppose that $r>0$.
Theorem~2.14 of~\cite{Mic} shows that the set of orbit representatives of $\GSp^{\la\xi\ra}_{2g}{(R_\rho)}=\Sp_{2g}{(R_\rho)}$ acting on $\Surj{\lp (R_\rho)^{2g},A\rp}$ injects into $\Lambda(A)$.
In fact, when $g\geq r$, this injection is a bijection.
Furthermore, this injection is equivariant with respect to the natural actions of $(R_\rho)^{\times}=\GSp_{2g}{(R_\rho)}/\Sp_{2g}{(R_\rho)}$ on these two sets.
Thus, computing the number of orbits of $\GSp^{\la\xi\ra}_{2g}{(R_\rho)}$ acting on $\Surj{\lp(R_\rho)^{2g},A\rp}$ is a straightforward application of Burnside's counting theorem.
Indeed, suppose that $A=\Z/\ell^{\alpha_1}\oplus\cdots\oplus\Z/\ell^{\alpha_r}$ with $\alpha_1\geq\cdots\geq\alpha_r>0$, then use \hyperref[forms]{Note~\ref*{forms}} to note that
\begin{align*}
o^{A,\la\xi\ra}_{2g,\rho}
&\leq\frac{1}{\lv\ell^\xi R_\rho+1\rv}
\cdot\sum_{\upsilon\in\ell^\xi R_\rho+1}{\lv\Fix{(\upsilon)}\rv}\\
&=\frac{1}{\lv\ell^\xi R_\rho+1\rv}\lp\lp\sum_{i=0}^{\rho-\xi-1}{
\sum_{\upsilon\in(\ell^{\xi+i}R_\rho+1)\setminus(\ell^{\xi+i+1}R_\rho+1)}
{\lv\Fix{(\upsilon)}\rv}}\rp+
\sum_{\upsilon\in\ell^\rho R_\rho+1=\{1\}}
{\lv\Fix{(\upsilon)}\rv}\rp\\
&=\frac{1}{\ell^{\rho-\xi}}
\lp\sum_{i=0}^{\rho-\xi-1}
{(\ell^{\rho-\xi-i}-\ell^{\rho-\xi-i-1})\lp
\ell^{\sum_{j=2}^{r}
{(j-1)\min{\lb\xi+i,\alpha_j\rb}}}\rp}
+\ell^{\alpha_2+2\alpha_3+\cdots+(r-1)\alpha_r}\rp\\
&=\frac{1}{\ell^{\rho-\xi}}
\lv\Lambda(A)\rv
+(\ell-1)\sum_{i=0}^{\rho-\xi-1}
{\ell^{-(i+1)}}{\lv\Lambda\lp A/\ell^{\xi+i}\rp\rv},
\end{align*}
with equality when $g\geq r$.
\end{proof}

\begin{Not3} \label{Not3}
Suppose that $A\in\mathcal{G}$, $\rho\in\Z^{>0}$, and $\xi\in\Z^{\geq0}$.
If $\ell^\rho\geq\exp{A}$ and $\rho\geq\xi$, use \hyperref[orb]{Lemma~\ref*{orb}} to define $o^{A,\la\xi\ra}_\rho:=o^{A,\la\xi\ra}_{2g,\rho}$ for any $g\in\Z^{>0}$ such that $g\geq\rank{A}$.
\end{Not3}

We now mention an identity which will be useful later.
(See \hyperref[moments]{Corollary~\ref*{moments}} and \hyperref[nutime]{Note~\ref*{nutime}}.)

\begin{subt} \label{subt}
Suppose $A\in\mathcal{G}$ and $\rho,\xi\in\Z^{>0}$.
If $\ell^\rho\geq\exp{A}$ and $\rho>\xi$, then by \hyperref[orb]{Lemma~\ref*{orb}} and \hyperref[forms]{Note~\ref*{forms}}, we see that
\[
\ell o^{A,\la\xi\ra}_\rho-o^{A,\la\xi+1\ra}_\rho
=(\ell-1)\lv\Lambda\lp A/\ell^\xi\rp\rv.
\]
\end{subt}


Below is a simple observation, which has the important consequence \hyperref[moments]{Corollary~\ref*{moments}}.
This corollary gives the moments of the probability distributions $\mu_x:\mathcal{G}\to\R$ for any $x\in\Z_\ell$, as promised in \hyperref[THM]{Section~\ref*{THM}}.

\begin{bij} \label{bij}
Suppose that $A\in\mathcal{G}$, $g,\rho\in\Z^{>0}$, and $\xi\in\Z^{\geq0}$.
Furthermore, suppose $\rho\geq\xi$, let $\gamma\in\GSp^{\la\xi\ra}_{2g}{(R_\rho)}$, and consider $\Inj{\lp A,\ker{(\Id-\gamma)}\rp}\subseteq\Inj{(A,(R_\rho)^{2g})}$.
There is a 1-1 correspondence between $\Inj{\lp A,\ker{(\Id-\gamma)}\rp}$ and $\Fix{(\gamma)}$.
Dually, $\Surj{\lp\coker{(\Id-\gamma)},A\rp}\leftrightarrow\Fix{(\gamma)}$.
\end{bij}
\begin{proof}
Suppose that $f\in\Inj{(A,(R_\rho)^{2g})}$.
Note that $f\in\Inj{\lp A,\ker{(\Id-\gamma)}\rp}$ if and only if $(\Id-\gamma)f=0$ if and only if $f=\gamma f$.
The dual proof is similar.
\end{proof}

\begin{moments} \label{moments}
Let $x\in\Z_\ell$ and suppose that $x\equiv1\pmod{\ell^\xi}$ but $x\nequiv1\pmod{\ell^{\xi+1}}$ for some $\xi\in\Z^{>0}$.
If $A\in\mathcal{G}$, then
\[
\sum_{B\in\mathcal{G}}{\lv\Surj{(B,A)}\rv\mu_x(B)}
=\lv\Lambda\lp A/\ell^\xi\rp\rv.
\]
\end{moments}
\begin{proof}
Choose any $g,\rho\in\Z^{>0}$ such that $g\geq\rank{A}$, $\ell^\rho\geq\exp{A}$, and $\rho>\xi$.
To begin with, note that
\begin{align*}
\sum_{B\in\mathcal{G}}
{\lv\Surj{\lp B,A\rp}\rv\nu_{2g,\rho}^{(x)}(B)}
&=\lv\GSp_{2g}^{(x)}{\lp R_\rho\rp}\rv^{-1}
\cdot\sum_{B\in\mathcal{G}}
{\lv\Surj{\lp B,A\rp}\rv
\cdot\lv\lb\gamma\in\GSp_{2g}^{(x)}{\lp R_\rho\rp}
\mid\coker{(\Id-\gamma)}\simeq B\rb\rv}\\
&=\lv\GSp_{2g}^{(x)}{\lp R_\rho\rp}\rv^{-1}
\cdot\sum_{\gamma\in\GSp_{2g}^{(x)}{\lp R_\rho\rp}}
{\lv\Surj{\lp\coker{(\Id-\gamma),A}\rp}\rv}.
\end{align*}
Now, thanks to \hyperref[subcos]{Lemma~\ref*{subcos}}, we can turn our attention to the following quantity:
\[
\lv\GSp_{2g}^{\la\xi\ra}{\lp R_\rho\rp}
\setminus\GSp_{2g}^{\la\xi+1\ra}{\lp R_\rho\rp}\rv^{-1}
\cdot\sum_{\gamma\in\GSp_{2g}^{\la\xi\ra}{\lp R_\rho\rp}
\setminus\GSp_{2g}^{\la\xi+1\ra}{\lp R_\rho\rp}}
{\lv\Surj{\lp\coker{\lp\Id-\gamma\rp},A\rp}\rv}.
\]
Using the fact that $\lv\GSp_{2g}^{\la\xi\ra}{(R_\rho)}\rv=\ell\lv\GSp_{2g}^{\la\xi+1\ra}{(R_\rho)}\rv$ and applying \hyperref[bij]{Lemma~\ref*{bij}} to $\GSp_{2g}^{\la\xi\ra}{(R_\rho)}$ acting on $\Surj{((R_\rho)^{2g},A)}$, then using Burnside's counting theorem and \hyperref[Not3]{Notation~\ref*{Not3}}, we see that
\begin{align*}
&\lv\GSp_{2g}^{\la\xi\ra}{\lp R_\rho\rp}
\setminus\GSp_{2g}^{\la\xi+1\ra}{\lp R_\rho\rp}\rv^{-1}
\cdot\sum_{\gamma\in\GSp_{2g}^{\la\xi\ra}{\lp R_\rho\rp}
\setminus\GSp_{2g}^{\la\xi+1\ra}{\lp R_\rho\rp}}
{\lv\Surj{\lp\coker{\lp\Id-\gamma\rp},A\rp}\rv}\\
&\qquad=\frac{\ell}
{(\ell-1)\lv\GSp_{2g}^{\la\xi\ra}{\lp R_\rho\rp}\rv}
\lp\sum_{\gamma\in\GSp_{2g}^{\la\xi\ra}{\lp R_\rho\rp}}
{\lv\Fix{\gamma}\rv}
-\sum_{\gamma\in\GSp_{2g}^{\la\xi+1\ra}{\lp R_\rho\rp}}
{\lv\Fix{\gamma}\rv}\rp\\
&\qquad=\frac{\ell}{(\ell-1)}
\lp\sum_{\gamma\in\GSp_{2g}^{\la\xi\ra}{\lp R_\rho\rp}}
{\frac{\lv\Fix{\gamma}\rv}{\lv\GSp_{2g}^{\la\xi\ra}{\lp R_\rho\rp}\rv}}
-\sum_{\gamma\in\GSp_{2g}^{\la\xi+1\ra}{\lp R_\rho\rp}}
{\frac{\lv\Fix{\gamma}\rv}{\ell\lv\GSp_{2g}^{\la\xi+1\ra}{\lp R_\rho\rp}\rv}}\rp\\
&\qquad=\frac{\ell}{(\ell-1)}
\lp o^{A,\la\xi\ra}_{2g,\rho}-\frac{1}{\ell}o^{A,\la\xi+1\ra}_{2g,\rho}\rp\\
&\qquad=\frac{1}{(\ell-1)}
\lp\ell o^{A,\la\xi\ra}_\rho-o^{A,\la\xi+1\ra}_\rho\rp,
\end{align*}
so we can conclude by applying \hyperref[subt]{Note~\ref*{subt}} and \hyperref[haar]{Lemma~\ref*{haar}}.
\end{proof}

\section{A weighted M\"{o}bius function} \label{AWMF}

\subsection{First observations}

Let $\mathcal{P}$ be a locally finite poset.
The \emph{M\"{o}bius function} on $\mathcal{P}$, denoted by $\mu_\mathcal{P}$, is defined by the following criteria: for any $x,z\in\mathcal{P}$,
\begin{align*}
\hspace{73px}&\mu_\mathcal{P}\lp x,z\rp&&\hspace{-83px}=0&&
\hspace{-73px}\text{ if }x\nleq z,\\
\hspace{73px}&\mu_\mathcal{P}\lp x,z\rp&&\hspace{-83px}=1&&
\hspace{-73px}\text{ if }x=z,\\
\hspace{73px}\sum_{x\leq y\leq z}
&{\mu_\mathcal{P}\lp x,y\rp}&&\hspace{-83px}=0&&
\hspace{-73px}\text{ if }x<z.
\end{align*}
A classic reference for M\"{o}bius functions is~\cite{Rota}.
In this section, we need to study a variant of the M\"{o}bius function on the poset of subgroups of a finite group (ordered by inclusion). 
For a history of the work on the M\"{o}bius function on this poset, see~\cite{HIO}.
Now, for any finite group $G$, let $\mathcal{P}_G$ be the poset of subgroups of $G$ ordered by inclusion.
For $A\in\mathcal{G}$, we study an amalgam of the M\"{o}bius functions on $\mathcal{P}_A$ and $\mathcal{G}$, which we define below.

\begin{Not4} \label{Not4}
For any $A,B\in\mathcal{G}$, let $\sub{(A,B)}$ be the number of subgroups of $B$ that are isomorphic to $A$.
If $A\in\mathcal{G}$, an \emph{$A$-chain} is a finite (possibly empty) linearly ordered subset of $\lb B\in\mathcal{G}\mid B>A\rb$.
Now, given an $A$-chain $\mathfrak{C}=\lb A_j\rb_{j=1}^i$, with $A_j<A_{j+1}$ for all $j\in\lb1,\ldots,i-1\rb$, define
\[
\sub{\lp\mathfrak{C}\rp}
:=(-1)^i\sub{\lp A,A_1\rp}\prod_{j=1}^{i-1}{\sub{\lp A_j,A_{j+1}\rp}}.
\]
(We set $\sub{\lp\mathfrak{C}\rp}=1$ if $\mathfrak{C}$ is empty.)
Finally, for any $A,B\in\mathcal{G}$, let
\[
S(A,B)
=\begin{cases}
\hspace{18px}0&\text{if }A\nleq B,\\
\hspace{18px}1&\text{if }A=B,\\
\displaystyle{
\sum_{\substack{A\text{-chains }\mathfrak{C},\\
\max{\mathfrak{C}}=B}}
{\sub{\lp\mathfrak{C}\rp}}}&\text{if } A<B.
\end{cases}
\]
\end{Not4}

\begin{amalgam} \label{amalgam}
Though $S$ is defined on the poset $\mathcal{G}$, it is closely related to the classical work on the M\"{o}bius function on the subgroup lattice of a fixed group.
Indeed, by applying Lemma~2.2 of~\cite{HIO}, we see that if $A,B\in\mathcal{G}$, then
\[
S(A,B)
=\sum_{\substack{C\leq B
\\C\simeq A}}
{\mu_B(C,B)}.
\]
\end{amalgam}

Given $x\in\Z_\ell$, we can use the function $S$ defined in \hyperref[Not4]{Notation~\ref*{Not4}} to begin our analysis of the measure $\mu_x$, following the outline in \hyperref[goa]{Goal~\ref*{goa}}.

\begin{burn} \label{burn}
Suppose $A\in\mathcal{G}$, $g,\rho\in\Z^{>0}$, and $\xi\in\Z^{\geq0}$, with $\rho\geq\xi$ and $\ell^\rho\geq\exp{A}$.
Then
\[
o^{A,\la\xi\ra}_{2g,\rho}\lv\GSp^{\la\xi\ra}_{2g}{(R_\rho)}\rv
=\sum_{\substack{B\in\mathcal{G}\\B\leq(R_\rho)^{2g}}}
{N^{\la\xi\ra}_{2g,\rho}(B)\lv\Inj{(A,B)}\rv}.
\]
\end{burn}
\begin{proof}
Applying \hyperref[bij]{Lemma~\ref*{bij}} and Burnside's counting theorem, we see that
\[
o^{A,\la\xi\ra}_{2g,\rho}\lv\GSp^{\la\xi\ra}_{2g}{(R_\rho)}\rv
=\sum_{\gamma\in\GSp^{\la\xi\ra}_{2g}{(R_\rho)}}{\lv\Fix{(\gamma)}\rv}
=\sum_{\gamma\in\GSp^{\la\xi\ra}_{2g}{(R_\rho)}}
{\lv\Inj{\lp A,\ker{(\Id-\gamma)}\rp}\rv}
=\sum_{\substack{B\in\mathcal{G}\\B\leq(R_\rho)^{2g}}}{N^{\la\xi\ra}_{2g,\rho}(B)\lv\Inj{(A,B)}\rv},
\]
where the last step follows from \hyperref[dual2]{Lemma~\ref*{dual2}}.
\end{proof}

For $A,g,\rho,\xi$ as above, \hyperref[burn]{Lemma~\ref*{burn}} gives us an ``upper triangular" system of equations, which we will solve for $N^{\la\xi\ra}_{2g,\rho}(A)$.
(The quotes indicate that the system is indexed by the poset $\mathcal{P}_{(r_\rho)^{2g}}$.)
\hyperref[deter]{Proposition~\ref{deter}} is the first step along this path.

\begin{deter} \label{deter}
Suppose $A,g,\rho,\xi$ are as above.
Then
\[
\frac{N^{\la\xi\ra}_{2g,\rho}(A)}{\lv\GSp_{2g}^{\la\xi\ra}{(R_\rho)}\rv}
=\sum_{\substack{B\in\mathcal{G}\\B\leq(R_\rho)^{2g}}}
{o^{B,\la\xi\ra}_{2g,\rho}\cdot \frac{S(A,B)}{\lv\Aut{B}\rv}}.
\]
\end{deter}
\begin{proof}
We use strong induction on $\lv(R_\rho)^{2g}\rv/\lv A\rv$.
In light of \hyperref[burn]{Lemma~\ref*{burn}}, the base case $A=(R_\rho)^{2g}$ is trivial.
Now suppose the result is true for all $B\in\mathcal{G}$ with $B\leq(R_\rho)^{2g}$ and $\lv(R_\rho)^{2g}\rv/\lv B\rv<\lv(R_\rho)^{2g}\rv/\lv A\rv$.
Using \hyperref[burn]{Lemma~\ref*{burn}}, we see that
\begin{align*}
\frac{N^{\la\xi\ra}_{2g,\rho}(A)}{\lv\GSp_{2g}^{\la\xi\ra}{(R_\rho)}\rv}
&=\frac{1}{\lv\Aut{A}\rv}\cdot\lp o^{A,\la\xi\ra}_{2g,\rho}
-\frac{1}{\lv\GSp^{\la\xi\ra}_{2g}{(R_\rho)}\rv}
\cdot\sum_{\substack{B\in\mathcal{G}\\B\leq(R_\rho)^{2g}\\B\neq A}}
{N^{\la\xi\ra}_{2g,\rho}(B)\lv\Inj{(A,B)}\rv}\rp\\
&=\frac{o^{A,\la\xi\ra}_{2g,\rho}}{\lv\Aut{A}\rv}
-\sum_{\substack{B\in\mathcal{G}\\B\leq(R_\rho)^{2g}\\B\neq A}}
{\frac{N^{\la\xi\ra}_{2g,\rho}(B)}{\lv\GSp_{2g}^{\la\xi\ra}{(R_\rho)}\rv}\cdot\sub{(A,B)}},
\end{align*}
so the result follows by the induction hypothesis.
\end{proof}

\subsection{Vanishing of the M\"{o}bius function}

Before proceeding, we need a bit more notation, and a result from~\cite{GMob}.

\begin{Not23} \label{Not23}
For any $A\in\mathcal{G}$ and any $i\in\Z^{\geq0}$, let
\[
A_{\oplus i}
:=A\oplus\overbrace{\lp\Z/\ell\rp\oplus\cdots\oplus\lp\Z/\ell\rp}^{i\text{ times}}.
\]
\end{Not23}

In 1934, Hall~\cite{Hall2} proved that if $G$ is an $\ell$-group of order $\ell^n$, then $\mu_G(1,G)=0$ unless $G$ is elementary abelian, in which case $\mu_G(1,G)=(-1)^n\ell^{\binom{n}{2}}$.
The paper~\cite{GMob} proves an analogous property of the function $S$:

\begin{finalm}[\cite{GMob}] \label{finalm}
If $A,B\in\mathcal{G}$, then $S(A,B)=0$ unless there exists an injection $\iota:A\hookrightarrow B$ with $\coker{(\iota)}$ elementary abelian.
\end{finalm}
\noindent \hyperref[finalm]{Theorem~\ref*{finalm}} has the following corollary:

\begin{infi} \label{infi}
Suppose $A,g,\rho,\xi$ are as above, and let $r=\rank{A}$.
If in addition we know $\xi\in\Z^{>0}$ and $\rho$ satisfies $\rho>\xi$ and $\ell^\rho>\exp{A}$, then
\[
\nu_{2g,\rho}^{\la\xi\ra}(A)
=\sum_{B\in\mathcal{G}(r)}
{S(A,B)\cdot\sum_{i=0}^{2g-r}
{\frac{\ell o^{B_{\oplus i},\la\xi\ra}_{2g,\rho}
-o^{B_{\oplus i},\la\xi+1\ra}_{2g,\rho}}
{\ell-1}
\cdot \frac{S(B,B_{\oplus i})}{\lv\Aut{B_{\oplus i}}\rv}}}.
\]
\end{infi}
\begin{proof}
Using the fact that $\lv\GSp_{2g}^{\la\xi\ra}{(R_\rho)}\rv=\ell\lv\GSp_{2g}^{\la\xi+1\ra}{(R_\rho)}\rv$ along with \hyperref[deter]{Proposition~\ref*{deter}} and \hyperref[finalm]{Theorem~\ref*{finalm}}, we see that
\begin{align*}
\nu_{2g,\rho}^{\la\xi\ra}(A)
&=\frac{N_{2g,\rho}^{\la\xi\ra}(A)
-N_{2g,\rho}^{\la\xi+1\ra}(A)}
{\lv\GSp^{\la\xi\ra}_{2g}{(R_\rho)}\rv
-\lv\GSp^{\la\xi+1\ra}_{2g}{(R_\rho)}\rv}\\
&=\frac{\ell N_{2g,\rho}^{\la\xi\ra}(A)}
{(\ell-1)\lv\GSp_{2g}^{\la\xi\ra}{(R_\rho)}\rv}
-\frac{N_{2g,\rho}^{\la\xi+1\ra}(A)}
{(\ell-1)\lv\GSp_{2g}^{\la\xi+1\ra}{(R_\rho)}\rv}\\
&=\sum_{B\in\mathcal{G}(r)}
{S(A,B)\cdot\sum_{i=0}^{2g-r}
{\frac{\ell o^{B_{\oplus i},\la\xi\ra}_{2g,\rho}
-o^{B_{\oplus i},\la\xi+1\ra}_{2g,\rho}}
{\ell-1}
\cdot\frac{S(B,B_{\oplus i})}{\lv\Aut{B_{\oplus i}}\rv}}}.
\end{align*}
\end{proof}

\begin{nutime} \label{nutime}
Suppose $A,g,\rho,\xi,r$ are as in \hyperref[infi]{Corollary~\ref*{infi}}, and suppose that $g\geq r$.
Then for any $B\in\mathcal{G}(r)$ and $i\in\lb0\ldots,g-r\rb$, we know by \hyperref[subt]{Note~\ref*{subt}} and \hyperref[forms]{Lemma~\ref*{forms}} that
\[
\frac{\ell o^{B_{\oplus i},\la\xi\ra}_{2g,\rho}
-o^{B_{\oplus i},\la\xi+1\ra}_{2g,\rho}}
{\ell-1}
=\lv\Lambda\lp B_{\oplus i}/\ell^\xi\rp\rv
=\ell^{ir+\frac{i(i-1)}{2}}\lv\Lambda\lp B/\ell^\xi B\rp\rv,
\]
and for any $i\in\lb g-r+1\ldots,2g-r\rb$, we can use the proof of \hyperref[orb]{Lemma~\ref*{orb}} to note that
\[
\frac{\ell o^{B_{\oplus i},\la\xi\ra}_{2g,\rho}
-o^{B_{\oplus i},\la\xi+1\ra}_{2g,\rho}}
{\ell-1}
\leq\ell o^{B_{\oplus i},\la\xi\ra}_{2g,\rho}
\leq\ell\lv\Lambda\lp B_{\oplus i}\rp\rv
=\ell^{ir+\frac{i(i-1)}{2}+1}\lv\Lambda\lp B\rp\rv.
\]
Thus, if
\[
\sum_{i=0}^\infty
{\ell^{ir+\frac{i(i-1)}{2}}\frac{S(B,B_{\oplus i})}{\lv\Aut{B_{\oplus i}}\rv}}
\]
converges absolutely (and it does, see \hyperref[product]{Lemma~\ref*{product}}, \hyperref[converge]{Lemma~\ref*{converge}}, and \hyperref[final]{Theorem~\ref*{final}}), then so does 
\[
\sum_{i=0}^\infty
{\frac{\ell o^{B_{\oplus i},\la\xi\ra}_{2g,\rho}
-o^{B_{\oplus i},\la\xi+1\ra}_{2g,\rho}}
{\ell-1}
\cdot \frac{S(B,B_{\oplus i})}{\lv\Aut{B_{\oplus i}}\rv}},
\]
and
\[
\lim_{g\to\infty}
{\nu_{2g,\rho}^{\la\xi\ra}(A)}
=\sum_{B\in\mathcal{G}(r)}
{S(A,B)\lv\Lambda\lp B/\ell^\xi B\rp\rv\cdot\sum_{i=0}^\infty
{\ell^{ir+\frac{i(i-1)}{2}}\frac{S(B,B_{\oplus i})}{\lv\Aut{B_{\oplus i}}\rv}}}.
\]
Analyzing the inner series is the subject of the next section.
(Note that the above limit does not depend on $\rho$, once $\rho$ is large enough; this is consistent with \hyperref[haar]{Lemma~\ref*{haar}}.)
\end{nutime}

\section{$q$-series and convergence} \label{QSAC}

\subsection{$q$-series} \label{QS}

Before continuing, we make a small foray into some $q$-series notations and calculations.

\begin{Not6} \label{Not6}
For $z,q\in\C$ with $\lv q\rv<1$ and $i\in\Z^{\geq0}$, let
\[
\lp z;q\rp_i:=\prod_{j=0}^{i-1}{\lp1-q^jz\rp}.
\]
To ease notation, set $(q)_i:=(q;q)_i$.
Recall the definition of the $q$-binomial coefficients: for any $k,m\in\Z^{\geq0}$, let
\[
\binom{k}{m}_q:=\frac{(q)_k}{(q)_m(q)_{k-m}},
\]
with $\binom{k}{m}_q:=0$ if $k<m$.

For $i\in\Z^{\geq0}$, let $r_i=-1/\lp\ell^{\frac{i(i+1)}{2}}(\ell^{-1})_i\rp$.
We define the next object in terms of any finite set of nonnegative integers $S$ and any $i\in\Z$ satisfying $i>\max{S}$.
If $S\cup\{0\}=\{s_0,\ldots,s_j\}$, where $0=s_0<s_1<\cdots<s_{j+1}:=i$, define $r^i_S=\prod_{i=0}^{j}{r_{s_{i+1}-s_i}}$.

Finally, let $t_0=1$, let $t_1=r^1_{\varnothing}$, and for $i>1$, let
\[
t_i=\sum_{S\subseteq\{1,\ldots i-1\}}{r^i_S}.
\]
\end{Not6}

\begin{product} \label{product}
\[
\sum_{i=0}^{\infty}{t_i}=\prod_{i=1}^{\infty}{\lp1+\ell^{-i}\rp^{-1}}.
\]
\end{product}
\begin{proof}
Let $R=r_1+r_2+\cdots$ and, to get into the spirit of a $q$-series calculation, let $q=\ell^{-1}$.
Using a product formula of Euler (see~\cite{And}, p~19), we note that
\[
R
=-\sum_{i=1}^{\infty}{\frac{q^{\frac{i(i+1)}{2}}}{(1-q^i)\cdots(1-q)}}
=-\sum_{i=1}^{\infty}{\frac{q^iq^{\frac{i(i-1)}{2}}}{(1-q^i)\cdots(1-q)}}
=1-\prod_{i=1}^{\infty}{(1+q^i)}.
\]
Now, by the definition of $t_i$ (and by using \hyperref[converge]{Lemma~\ref*{converge}} to rearrange the terms of the sum), we know
\[
\sum_{i=0}^{\infty}{t_i}
=1+R+R^2+R^3+\cdots
=\frac{1}{1-R}
=\prod_{i=1}^{\infty}{\lp1+\ell^{-i}\rp^{-1}},
\]
as desired.
\end{proof}

Next, we justify the reordering of the summands in \hyperref[product]{Lemma~\ref*{product}}:

\begin{converge} \label{converge}
For any finite set of nonnegative integers $S$ and $i\in\Z$ satisfying $i>\max{S}$, let $\rho^i_S:=\lv r^i_S\rv$.
Next, let $\tau_0=1$, let $\tau_1=\rho^1_{\varnothing}$, and for any $i>1$, let
\[
\tau_i:=\sum_{S\subseteq\{1,\ldots i-1\}}{\rho^i_S}.
\]
Then $\sum_{i=0}^{\infty}{\tau_i}$ converges.
\end{converge}
\begin{proof}
For fun, we will give two proofs: a simple proof that holds for any $\ell>3$, and a more complicated one that holds for $\ell\geq3$.
Note that the sum clearly diverges for $\ell=2$ since it includes infinitely many 1's.

For the simple proof, note that for any finite set $S$ of nonnegative integers and any $i>\max{S}$, we know $\rho^i_S\leq(\ell-1)^{-i}$.
It follows that for any $i\in\Z^{\geq0}$, we have that $\tau_i\leq2^{i-1}(\ell-1)^{-i}$, so $\sum_{i=0}^{\infty}{\tau_i}$ converges for $\ell>3$.

Of course, the above argument fails for $\ell=3$.
In this case, for a finite set $S$ of nonnegative integers and an $i>\max{S}$, we must use a (slightly) better bound than $\rho^i_S\leq(\ell-1)^{-i}$.
Let $\lambda=(\ell-1)^{-1}$.
Since $(\ell^m-1)^{-1}\leq(\ell-1)^{-m}$ for any $m\in\Z^{\geq0}$, if we let $S\cup\{0\}=\{s_0,\ldots,s_j\}$, where $0=s_0<s_1<\cdots<s_{j+1}:=i$, then
\begin{equation} \label{trianglebound}
\rho^i_S
=\prod_{k=0}^{j}{\lv r_{s_{k+1}-s_k}\rv}
\leq\prod_{k=0}^{j}{\lambda^{\frac{1}{2}(s_{k+1}-s_k)(s_{k+1}-s_k+1)}}.
\end{equation}
Let $T_i$ be the number of compositions of $i$ by triangular numbers.
Then by rearranging the terms of $\sum_{i=0}^{\infty}{\tau_i}$ to order them by the exponent of $\lambda$ appearing in the bound~\hyperref[trianglebound]{(\ref*{trianglebound})}, we see that if $\sum_{i=1}^{\infty}{T_i\lambda^i}$ converges, then so does $\sum_{i=0}^{\infty}{\tau_i}$.
Since the generating function for the number of compositions of positive triangular numbers is
\begin{equation} \label{gftricomps}
\sum_{i=0}^{\infty}{T_ix^i}=\frac{1}{1-\lp\sum_{j=1}^{\infty}{x^{\frac{1}{2}j(j+1)}}\rp},
\end{equation}
we need only show that the radius of convergence of~\hyperref[gftricomps]{(\ref*{gftricomps})} is at least $\lambda$.
Since $\ell\geq3$, we know that $\lambda\leq\frac{1}{2}$, and
\[
1>\frac{1}{2}+\lp\frac{1}{2}\rp^3+\lp\frac{1}{2}\rp^6+\lp\frac{1}{2}\rp^{10}+\cdots,
\]
so the lemma is true.
\end{proof}

We can now finish proving the result mentioned in \hyperref[nutime]{Note~\ref*{nutime}}.

\begin{final} \label{final}
Suppose $A\in\mathcal{G}$, $\rho,\xi\in\Z^{>0}$, and let $r=\rank{A}$.
If $\rho>\xi$ and $\ell^\rho>\exp{A}$, then
\[
\lim_{g\to\infty}
{\nu_{2g,\rho}^{\la\xi\ra}(A)}
=\prod_{i=1}^{\infty}{\lp1+\ell^{-i}\rp^{-1}}\cdot\sum_{B\in\mathcal{G}(r)}
{\lv\Lambda\lp B/\ell^\xi B\rp\rv\cdot\frac{S(A,B)}{\lv\Aut{B}\rv}}.
\]
\end{final}
\begin{proof}
Let $B\in\mathcal{G}(r,s)$, let $S$ be a finite set of nonnegative integers, and let $i$ a positive integer with $i>\max{S}$.
Suppose $S\cup\lb0\rb=\lb s_0,\ldots,s_j\rb$, where $0=s_0<\cdots<s_{j+1}:=i$.
Now, we know by~\cite{GAb} that for any $k,m\in\Z^{\geq0}$ with $k\leq m$
\[
\sub{\lp B_{\oplus k},B_{\oplus m}\rp}
=\frac{\ell^{(r+k)(m-k)}\lp\ell^{-1}\rp_{r-s+m}}
{\lp\ell^{-1}\rp_{r-s+i}\lp\ell^{-1}\rp_{m-k}}
\]
and
\[
\lv\Aut{B_{\oplus i}}\rv
=\frac{\ell^{2ir+i^2}\lp\ell^{-1}\rp_{r-s+i}}{\lp\ell^{-1}\rp_{r-s}}\lv\Aut{B}\rv,
\]
so
\begin{align*}
(-1)^{j+1}\cdot\frac{\ell^{ir+\frac{i(i-1)}{2}}}{\lv\Aut{B_{\oplus i}}\rv}
\cdot\prod_{k=0}^j
{\sub{\lp B_{\oplus s_k},B_{\oplus s_{k+1}}\rp}}
&=(-1)^{j+1}
\cdot\frac{\ell^{-ir-\frac{i(i+1)}{2}}}{\lv\Aut{B}\rv}
\cdot\frac{\lp\ell^{-1}\rp_{r-s}}{\lp\ell^{-1}\rp_{r-s+i}}
\cdot\prod_{k=0}^j
{\frac{\ell^{(r+s_k)(s_{k+1}-s_k)}\lp\ell^{-1}\rp_{r-s+s_{k+1}}}
{\lp\ell^{-1}\rp_{r-s+s_k}\lp\ell^{-1}\rp_{s_{k+1}-s_k}}}\\
&=(-1)^{j+1}
\cdot\frac{\ell^{-ir-\frac{i(i+1)}{2}}}{\lv\Aut{B}\rv}
\cdot\prod_{k=0}^j
{\frac{\ell^{(r+s_k)(s_{k+1}-s_k)}}
{\lp\ell^{-1}\rp_{s_{k+1}-s_k}}}\\
&=(-1)^{j+1}
\cdot\frac{1}{\lv\Aut{B}\rv}
\cdot\prod_{k=0}^j
{\frac{\ell^{-\frac{1}{2}\lp s_{k+1}-s_k\rp\lp s_{k+1}-s_k+1\rp}}
{\lp\ell^{-1}\rp_{s_{k+1}-s_k}}}.
\end{align*}
But by \hyperref[product]{Lemma~\ref*{product}}, this means that 
\[
\sum_{i=0}^\infty
{\ell^{ir+\frac{i(i-1)}{2}}\frac{S(B,B_{\oplus i})}{\lv\Aut{B_{\oplus i}}\rv}}
=\frac{1}{\lv\Aut{B}\rv}\cdot\sum_{i=0}^\infty{t_i}
=\frac{1}{\lv\Aut{B}\rv}\cdot\prod_{i=1}^{\infty}{\lp1+\ell^{-i}\rp^{-1}},
\]
so we conclude by \hyperref[nutime]{Note~\ref{nutime}}.
\end{proof}

\subsection{The main results}

To conclude we mention two corollaries of \hyperref[final]{Theorem~\ref*{final}}, one trivial and one nontrivial.

\begin{trivial} \label{trivial}

If $x\in\Z_\ell$ satisfies $x\equiv1\pmod{\ell}$, then
\[
\lim_{g\to\infty}
{\mu_{2g}^{(x)}(\lb0\rb)}
=\prod_{i=1}^{\infty}{\lp1+\ell^{-i}\rp^{-1}}.
\]
\end{trivial}

The nontrivial corollary relies heavily on calculations from~\cite{GAb}.

\begin{finalc} \label{finalc}
Suppose $r,s\in\Z^{\geq0}$ with $r\geq s$.
Furthermore, suppose that $x\in\Z_\ell$ and $\xi\in\Z^{>0}$ with $x\equiv1\pmod{\ell^\xi}$ but $x\nequiv1\pmod{\ell^{\xi+1}}$.
If $A\in\mathcal{G}(r,s)$, then
\[
\lim_{g\to\infty}
{\mu_{2g}^{(x)}(A)}
=\begin{cases}
\ell^{\frac{r(r-1)}{2}}\cdot\lp\ell^{-1}\rp_r
\cdot\frac{\prod_{i=1}^{\infty}{(1+\ell^{-i})^{-1}}}{\lv\Aut{A}\rv},
&\text{if }\xi=1\\
\ell^{\frac{r(r-1)}{2}+\frac{s(s-1)}{2}}
\cdot\lp\ell^{-1}\rp_s\lp\ell^{-1};\ell^{-2}\rp_{\lceil \frac{r-s}{2}\rceil}
\cdot\frac{\prod_{i=1}^{\infty}{(1+\ell^{-i})^{-1}}}{\lv\Aut{A}\rv},
&\text{if }\xi=2.
\end{cases}
\]
\end{finalc}
\begin{proof}
Choose any $\rho\in\Z^{>0}$ with $\rho>\xi$ and $\ell^\rho>\exp{A}$.
Then by \hyperref[haar]{Lemma~\ref*{haar}} we know
\[
\mu_{2g}^{(x)}(A)=\nu_{2g,\rho}^{\la\xi\ra}(A).
\]
Now, we know from~\cite{GAb} that
\[
\sum_{B\in\mathcal{G}(r)}
{\frac{S(A,B)}{\lv\Aut{B}\rv}}
=\frac{\lp\ell^{-1}\rp_r}{\lv\Aut{A}\rv},
\]
and for any $i\in\lb s,\ldots,r\rb$:
\[
\sum_{B\in\mathcal{G}(r,i)}
{\frac{S(A,B)}{\lv\Aut{B}\rv}}
=(-1)^{i-s}\cdot\ell^{\frac{s(s+1)}{2}-\frac{i(i+1)}{2}}\cdot
\binom{r-s}{r-i}_{\ell^{-1}}
\cdot\frac{\lp\ell^{-1}\rp_s}
{\lv\Aut{A}\rv}.
\]
The $\xi=1$ case follows from \hyperref[forms]{Note~\ref*{forms}}.
For $\xi=2$, use \hyperref[forms]{Note~\ref*{forms}} again to see that
\begin{align*}
\sum_{B\in\mathcal{G}(r)}
{\lv\Lambda\lp B/\ell^2B\rp\rv
\cdot\frac{S(A,B)}{\lv\Aut{B}\rv}}
&=\sum_{i=s}^r
{\sum_{B\in\mathcal{G}(r,i)}
{\lv\Lambda\lp B/\ell^2B\rp\rv
\cdot\frac{S(A,B)}{\lv\Aut{B}\rv}}}\\
&=\sum_{i=s}^r
{(-1)^{i-s}\cdot\ell^{\frac{r(r-1)}{2}+\frac{s(s+1)}{2}-i}
\cdot\binom{r-s}{r-i}_{\ell^{-1}}
\cdot\frac{\lp\ell^{-1}\rp_s}
{\lv\Aut{A}\rv}}\\
&=\frac{\ell^{\frac{r(r-1)}{2}+\frac{s(s+1)}{2}}\lp\ell^{-1}\rp_s}
{\lv\Aut{A}\rv}
\cdot\sum_{i=s}^r
{(-1)^{i-s}\cdot\binom{r-s}{r-i}_{\ell^{-1}}\cdot\ell^{-i}}.
\end{align*}
Letting $k=r-s$ and $q=1/\ell$, we apply formula (1.10) from~\cite{Kup}, which is a corollary of formula (1.12), to obtain
\begin{align*}
\sum_{B\in\mathcal{G}(r)}
{\lv\Lambda\lp B/\ell^2B\rp\rv\cdot\frac{S(A,B)}{\lv\Aut{B}\rv}}
&=\frac{\ell^{\frac{r(r-1)}{2}+\frac{s(s-1)}{2}}\lp\ell^{-1}\rp_s}
{\lv\Aut{A}\rv}
\cdot\sum_{i=0}^k{(-1)^i\binom{k}{i}_qq^i}\\
&=\frac{\ell^{\frac{r(r-1)}{2}+\frac{s(s-1)}{2}}\lp\ell^{-1}\rp_s}
{\lv\Aut{A}\rv}
\cdot\lp q;q^2\rp_{\lceil \frac{k}{2}\rceil},
\end{align*}
as desired.
\end{proof}

\bibliography{cohenlenstra}
\bibliographystyle{amsalpha}

\end{document}